\def\misajour{29/11/2019} 
 
 \documentclass[11pt,leqno]{article} 
 
\usepackage[applemac]{inputenc}
\usepackage[pdftex]{graphicx}
\usepackage{apalike}
 \usepackage{hyperref} 
 \usepackage{xy}
\usepackage{xypic} 
 
\usepackage{multicol}
\usepackage{xcolor}
\usepackage{amsmath}
\usepackage{amssymb}
\usepackage{amsthm} 
\usepackage{mathrsfs} 
\def\C{{\mathbb C}} 
\def\N{{\mathbb N}}
\def\Q{{\mathbb Q}}

\def\Z{{\mathbb Z}}
\def\scrS{{\mathscr S}}

\def\kappatilde{\tilde{\kappa}}
\def\ftilde{\tilde{f}}
\def\fhat{\hat{f}}

\def\Mtilde{\widetilde{M}}

\def\Lambdatilde{\widetilde{\Lambda}}

\def\rmd{{\mathrm d}}
\def\rme{{\mathrm e}}

\numberwithin{equation}{section}

\newtheorem{theorem}{Theorem}
\newtheorem{corollary}{Corollary}
\newtheorem{proposition}{Proposition}
\newtheorem{lemma}{Lemma}
\newtheorem*{remark}{Remark}

\begin{document} 
 
 \null
 \vskip -3 true cm
 
 {\footnotesize 
 \hfill
Update: {\it \misajour}
}

 \vskip 3 true cm

 \begin{center}
 \LARGE
 \bf 
 
 On transcendental entire functions 
 
 with infinitely many derivatives 
 taking 
 
 integer values at 
two points

 \large
 
 \medskip
 by
 
 \medskip
 \sc
 Michel Waldschmidt
 \footnote{
 \\
 {\sc Acknowledgments}. --- On November 23, 2018, during the International Conference on Special Functions \&\ Applications (ICSFA-2018) which took place in 
 Amal Jyothi College on Engineering, Kanjirapalli, Kottayam (Kerala, India), M.A. Pathan gave a talk \emph{On Generalization of Taylor's series, Riemann Zeta Functions and Bernoulli Polynomials}, where the author became acquainted with Lidstone series. 
 }

 \end{center}
 
\section*{Abstract}

Given a subset $S=\{s_0, s_1\}$ of the complex plane with two points and an infinite subset $\scrS$ of $S\times \N$, where $\N=\{0,1,2,\dots\}$ is the set of nonnegative integers, we ask for a lower bound for the order of growth of a transcendental entire function $f$ such that $f^{(n)}(s)\in\Z$ for all $(s,n)\in\scrS$. 
 
We first take 
$\scrS=\{s_0,s_1\}\times 2\N$, where $2\N=\{0,2,4,\dots\}$ is the set of nonnegative even integers.
We prove that an entire function $f$ of sufficiently small exponential type such that $f^{(2n)}(s_0)\in\Z$ and $f^{(2n)}( s_1)\in\Z$ for all sufficiently large $n$ must be a polynomial. The estimate we reach is optimal, as we show by constructing a noncountable set of examples. The main tool, both for the proof of the estimate and for the construction of examples, is Lidstone polynomials.

Our second example is 
$(\{s_0\}\times (2\N+1))\cup( \{ s_1\}\times 2\N)$ (odd derivatives at $s_0$ and even derivatives at $ s_1$). We use analogs of Lidstone polynomials which have been introduced by J.M.~Whittaker and studied by I.J.~Schoenberg.

Finally, using results of W.~Gontcharoff, A. J.~Macintyre and J.M.~Whittaker, we prove lower bounds for the exponential type of a transcendental entire function $f$ such that, for each sufficiently large $n$, one at least of the two numbers $f^{(n)}(s_0)$, $f^{(n)}(s_1)$ is in $\Z$.

\tableofcontents

\section{Introduction}\label{S:Introduction}

The order of an entire function $f$ is 
$$
\varrho(f)=\limsup_{r\to\infty} \frac{\log\log |f|_r}{\log r}
\;
\text{ where } 
\;
|f|_r=\sup_{|z|=r}|f(z)|.
$$
The exponential type of an entire function is 
$$
\tau(f)= \limsup_{r\to\infty} \frac{ \log |f|_r}{r}\cdotp
$$
If the exponential type is finite, then $f$ has order $\le 1$. If $f$ has order $<1$, then the exponential type is $0$.

An alternative definition 
 is the following: $f$ is of exponential type $\tau(f)$ if and only if, for all $z_0\in\C$, 
\begin{equation}\label{eq:type}
\limsup_{n\to\infty} |f^{(n)}(z_0)|^{1/n}= \tau(f),
\end{equation}
where $f^{(n)}$ denotes the $n$--th derivative $(\rmd /\rmd z)^n f$ of $f$. 
The equivalence between the two definitions follows from Cauchy's inequalities \eqref{Equation:CauchyInequality} and Stirling's Formula \eqref{Equation:Stirling}. If \eqref{eq:type} is true for one $z_0\in\C$, then it is true for all $z_0\in\C$.

Given a finite set of points $S$ in the complex plane and an infinite subset $\scrS$ of $S\times \N$, where $\N=\{0,1,2,\dots\}$ is the set of nonnegative integers, we ask for a lower bound for the order of growth of a transcendental entire function $f$ such that $f^{(n)}(s)\in\Z$ for all $(s,n)\in\scrS$. This question has been studied by a number of authors in the special case where $\scrS=S\times \N$. When $S=\{0\}$, a function satisfying these conditions, namely $f^{(n)}(0)\in\Z$ for all $n\ge 0$, is called a \emph{Hurwitz function}. The order of a transcendental Hurwitz function is $\ge 1$ \cite{zbMATH02601316} 
-- see Proposition \ref{Prop:Polya} below. 
Assume now $S=\{0,1,\dots,k-1\}$ with $k\ge 2$. 
According to \cite[Th.~1]{MR0035822}, 
the order of a transcendental function satisfying $f^{(n)}(\ell)\in\Z$ for all $\ell=0,1,\dots,k-1$ and $n\ge 0$ is at least $k$. The example of the function $\exp\bigl(z(z-1)\cdots(z-k+1)\bigr)$ shows that the bound for the order is sharp.
For $k=2$, refined estimates are obtained in  
\cite[\S~3]{MR0159945} 
and \cite[\S~4]{MR0192030}. 
See also 
\cite[\S~7 and \S~8]{MR0280715} 
and the survey \cite{MR789191} 
with 59 references.

If we replace the assumption $f^{(n)}(s)\in\Z$ %
with $f^{(n)}(s)=0$ for all $(s,n)\in\scrS$, we come across a question which has been the object of extensive works. It is the main topic of 
\cite[Chap.~III]{zbMATH02532117} 
and
\cite[Chap.~3]{zbMATH03416363}. 
It is related with the interpolation problem of the existence and unicity of an entire function $f$ for which the values $f^{(n)}(s)$ for $(s,n)\in\scrS$ are given. For $S=\{0\}$, the Taylor expansion solves the interpolation problem. The next most often studied case is $S=\{0,1\}$ and $\scrS=S\times 2\N$, where $\N=\{0,2,4,\dots\}$ is the set of nonnegative even integers; the basic tool is given by Lidstone polynomials.

In the present paper, we consider a set $S=\{s_0 , s_1\}$ of only two complex numbers (with only a short excursion to the case where $S$ may have more than two points in \S~\ref{Th:GontcharoffMacintyre} and \S~\ref{SS:GontcharoffMacintyre}. We plan to investigate more general sets in later papers). Using an argument going back to P\'olya, we reduce the study of entire functions, the derivatives of even order of which take integer values at two points, to the study of those functions where the same derivatives vanish at these two points. Our main assumption on the growth of our functions $f$ is
\begin{equation}\label{eq:maingrowthcondition}
\limsup_{r\to\infty}\rme^{-r}\sqrt r |f|_r<\frac{1}{\sqrt{2\pi}}
\rme^{-\max\{|s_0|,| s_1|\}}.
\end{equation}
The exponential type of such a function is $\le 1$; in the other direction, a function of exponential type $<1$ satisfies \eqref{eq:maingrowthcondition}. 
We will prove in \S~\ref{S:GeneralizationPolya}
that, for an entire function $f$ satisfying the growth condition \eqref{eq:maingrowthcondition} and for $|z_0|\le \max\{|s_0|,| s_1|\}$, the set of $n\ge 0$ for which $f^{(2n)}(z_0)\in\Z\setminus\{0\}$ is finite.

In \S~\ref{S:Lidstone}, 
 we introduce the so--called {\it Lidstone polynomials} and we prove several estimates for their growth.
 
 In \S~\ref{S:Even}, we give a lower bound for the growth of transcendental entire functions satisfying $f^{(2n)}(s_0)\in\Z$ and $f^{(2n)}( s_1)\in\Z$ for all sufficiently large $n$. 
 On the other hand, we will also show 
 that there are transcendental entire functions $f$ of order $0$ for which $f^{(2n)}(s_0)=0$ for all $n\ge 0$ and $f^{(2n)}( s_1)=0$ for infinitely many $n$.

In \S~\ref{S:OddEven}, we consider a variant by studying 
entire functions which satisfy 
$f^{(2n+1)}(s_0)\in\Z$ and $f^{(2n)}( s_1)\in\Z$ for all sufficiently large $n$. The proofs rest on analogs of Lidstone polynomials which have been introduced 
by J.M.~Whittaker in 1933
and studied by 
 I.J.~Schoenberg in 1936.

In \S~\ref{S:SequencesDerivatives}, we give a lower bound for the growth of transcendental entire functions satisfying 
the property that for each sufficiently large $n$, one at least of the two numbers $f^{(n)}(s_0)$, $f^{(n)}(s_1)$ is in $\Z$. 
In the periodic case we use results of W.~Gontcharoff (1930) and A.J.~Macintyre (1954), in the general case we use results of W.~Gontcharoff (1930) and J.M.~Whittaker (1933).
 
\subsection{Derivatives of even order at two points}

Our first result is a lower bound for the growth of a transcendental entire function, having derivatives of even order at two points $s_0$ and $ s_1$ in $\Z$.
 
\begin{theorem}\label{Th:TwoPointsEven}
Let $s_0, s_1$ be two distinct complex numbers and $f$ an entire function of exponential type $\tau(f)$ satisfying $f^{(2n)}(s_0)\in\Z$ and $f^{(2n)}( s_1)\in\Z$ for all sufficiently large $n$. Assume $f$ satisfies the growth condition \eqref{eq:maingrowthcondition}.
Then there exists a polynomial $P\in\C[z]$ and complex numbers $c_1,c_2,\dots,c_L$ with 
$$
L \pi \le | s_1-s_0| \tau(f)
$$
such that 
$$
f(z)=P(z)+\sum_{\ell=1}^L c_\ell \sin\left(\ell\pi\frac{z-s_0}{ s_1-s_0}
\right).
$$
\end{theorem}

Recall that assumption \eqref{eq:maingrowthcondition} implies $\tau(f)\le 1$. It follows from Theorem \ref{Th:TwoPointsEven} that, if $| s_1-s_0|\le \pi$, then any transcendental entire function $f$ satisfying $f^{(2n)}(s_0)\in\Z$ and $f^{(2n)}( s_1)\in\Z$ for all sufficiently large $n$ has exponential type $\ge 1$. Here are examples of such functions of exponential type $1$. Let $a_0\in\Z$ and $a_1\in\Z$ with $(a_0,a_1)\not=(0,0)$. Define
$$
f_{a_0,a_1}(z)=a_0\frac{\sinh(z- s_1)}{\sinh(s_0- s_1)}+
a_1\frac{\sinh(z- s_0)}{\sinh(s_1- s_0)}\cdotp
$$
Then $f_{a_0,a_1} (s_0)=a_0$, $f_{a_0,a_1} ( s_1)=a_1$ and $f_{a_0,a_1}''=f_{a_0,a_1}$, hence $f_{a_0,a_1}^{(2n)}(s_0)=a_0$ and $f_{a_0,a_1}^{(2n)}( s_1)=a_1$ for all $n\ge 0$.

In the case $| s_1-s_0|\ge \pi$, we deduce from Theorem \ref{Th:TwoPointsEven} that any transcendental entire function $f$ satisfying $f^{(2n)}(s_0)\in\Z$ and $f^{(2n)}( s_1)\in\Z$ for all sufficiently large $n$ has exponential type $\ge \pi/| s_1-s_0|$. 
For $\ell\ge 1$, the function 
$$
f_\ell(z)=\sin\left(\ell\pi\frac{z-s_0}{ s_1-s_0}\right)
$$
has exponential type $\ell \pi/| s_1-s_0|$ and satisfies $f_\ell^{(2n)}(s_0)=f^{(2n)}( s_1)=0$ for all $n\ge 0$.

\begin{corollary}\label{Corollary:twopointseven}
Let $f$ be an entire function satisfying \eqref{eq:maingrowthcondition} for which $f^{(2n)}(s_0)\in\Z$ and $f^{(2n)}( s_1)\in\Z$ for all sufficiently large $n$. Then the set of $n\ge 0$ such that $f^{(2n)}(s_0)\not=0$ is finite, and also the set of $n\ge 0$ such that $f^{(2n)}( s_1)\not=0$ is finite.
If the exponential type of $f$ satisfies $\tau(f)<\frac{\pi}{| s_1-s_0|}$, then $f$ is a polynomial. 
\end{corollary}

We now show hat the assumption \eqref{eq:maingrowthcondition} on the growth of $f$ in Corollary \ref{Corollary:twopointseven} is essentially best possible. 
We denote by $\nu$ the unique positive real number satisfying $e^\nu+e^{-\nu}=4\nu$. The numerical value is $\nu = 2.177\, 3\dots$ Both $\nu$ and $\rme^\nu$ are transcendental.

\begin{theorem}\label{Th:TwoPointsEvenExistence}
Let $s_0, s_1$ be two distinct complex numbers such that
\begin{equation}\label{equation:majoration|s0-s1|even}
|s_1-s_0|<\nu.
\end{equation}
Then there exist a constant $\gamma$ and an uncountable set of transcendental entire functions $f$ 
satisfying $f^{(2n)}(s_0)=0$ and $f^{(2n)}( s_1)\in\Z$ for all $n\ge 0$, for which the set 
$\{
n\ge 0\; \mid \; f^{(2n)}( s_1)\not=0\}
$
is infinite, and such that 
\begin{equation}\label{eq:maingrowthconditiongamma}
\limsup_{r\to\infty}\rme^{-r}\sqrt r |f|_r\le \gamma
\end{equation}
\end{theorem}
 
The conclusion is also true in the case $| s_1-s_0|> \pi$. Indeed, it follows from Theorem \ref{Th:TwoPointsEven} that in this case, if $f$ is an entire function of exponential type $<1$ satisfying $f^{(2n)}(s_0)\in\Z$ and $f^{(2n)}( s_1)\in\Z$ for all sufficiently large $n$,
 then the sets $\{n\ge 0\;\mid\;f^{(2n)}(s_0)\not=0\}$ and $\{n\ge 0\;\mid\;f^{(2n)}(s_1)\not=0\}$ are finite. Hence the set of these functions is the set of functions of the form
$$
f(z)=P(z)+\sum_{\ell=1}^L c_\ell \sin\left(\ell\pi\frac{z-s_0}{ s_1-s_0}
\right),
$$
where $P\in\C[z]$, 
$L$ is the larger integer $< |s_1-s_0|/\pi$ and $c_1,\dots,c_L$ are arbitrary complex numbers. Since $L\ge 1$, this set of functions $f$ has the power of continuum.

One might expect that it suffices to assume $|s_1-s_0|<\pi$ instead of \eqref{equation:majoration|s0-s1|even}. Another question is about the case $| s_1-s_0|=\pi$: the function $\sin z$ has exponential type $1$ but does not satisfy \eqref{eq:maingrowthconditiongamma}.

 From Corollary \ref{Corollary:twopointseven} one deduces the same result for the set $S\times(2\N+1)$ (odd order of the derivatives): it suffices to use Corollary \ref{Corollary:twopointseven} for the first derivative of the given function. 
 
 \begin{corollary}\label{Corollary:oddodd}
 Let $f$ be an entire function satisfying \eqref{eq:maingrowthcondition} for which $f^{(2n+1)}(s_0)\in\Z$ and $f^{(2n+1)}( s_1)\in\Z$ for all sufficiently large $n$. Then the set of $n\ge 0$ such that $f^{(2n+1)}(s_0)\not=0$ is finite, and also the set of $n\ge 0$ such that $f^{(2n+1)}( s_1)\not=0$ is finite.
If the exponential type of $f$ satisfies $\tau(f)<\frac{\pi}{| s_1-s_0|}$, then $f$ is a polynomial. 
 \end{corollary} 
 
A polynomial is determined only up to an additive constant by its derivatives of odd order at two points. An expansion of a polynomial in terms of these derivatives, analogous to \eqref{Equation:LidstoneExpansionPolynomials} below, is obtained by taking primitives of the Lidstone polynomials (defined up to an additive constant -- notice that $\Lambda'_{n+1}$ is a primitive of $\Lambda_n$). Such expansions have been studied in 
 \cite[\S~3]{MR3849168} 
 under the name \emph{Even Lidstone--type sequences}. 
 
\subsection{Derivatives of odd order at one point and even at the other}

The next result deals with $\scrS=(\{s_0\}\times (2\N+1))\cup (\{ s_1\}\times 2\N)$.

\begin{theorem}\label{Th:TwoPointsOddEven}
Let $s_0, s_1$ be two distinct complex numbers. Let $f$ be an entire function of exponential type $\tau(f)$ satisfying $f^{(2n+1)}(s_0)\in\Z$ and $f^{(2n)}( s_1)\in\Z$ for all sufficiently large $n$. Assume $f$ satisfies \eqref{eq:maingrowthcondition}.
Then there exists a polynomial $P\in\C[z]$ and complex numbers $c_0,c_1,\dots,c_L$ with 
$$
(2L+1)\frac \pi 2 \le | s_1-s_0| \tau(f)
$$
such that 
$$
f(z)=P(z)+\sum_{\ell=0}^L c_\ell \cos\left(\frac{(2\ell+1)\pi}{2}\cdot \frac{z-s_0}{ s_1-s_0}
\right).
$$
\end{theorem}

In the case $| s_1-s_0|\le \pi/2$, any transcendental entire function $f$ satisfying $f^{(2n+1)}(s_0)\in\Z$ and $f^{(2n)}( s_1)\in\Z$ for all sufficiently $n$ has exponential type $\ge 1$. Here are examples of such functions of exponential type $1$. Let $a_0\in\Z$ and $a_1\in\Z$ with $(a_0,a_1)\not=(0,0)$. Define
$$
f_{a_0,a_1}(z)=
a_0\frac{\cosh(z- s_1)}{\cosh(s_0- s_1)}
+
a_1\frac{\cosh(z-s_0)}{\cosh(s_1-s_0)}\cdotp
$$
Then $f'_{a_0,a_1} (s_0)=a_0$, $f_{a_0,a_1} ( s_1)=a_1$ and $f_{a_0,a_1}''=f_{a_0,a_1}$, hence $f_{a_0,a_1}^{(2n+1)}(s_0)=a_0$ and $f_{a_0,a_1}^{(2n)}( s_1)=a_1$ for all $n\ge 0$.

In the case $| s_1-s_0|\ge \pi/2$, any transcendental entire function $f$ satisfying $f^{(2n+1)}(s_0)\in\Z$ and $f^{(2n)}( s_1)\in\Z$ for all sufficiently large $n$ has exponential type $\ge \pi/(2| s_1-s_0|)$. 
For $\ell\ge 0$, the function
$$
f_\ell(z)=\cos\left(\frac{(2\ell+1)\pi}{2}\cdot \frac{z-s_0}{ s_1-s_0}
\right)
$$
has exponential type $\frac{(2\ell+1)\pi}{2| s_1-s_0|}$ and satisfies $f_\ell^{(2n+1)}(s_0)=f_\ell^{(2n)}( s_1)=0$ for all $n\ge 0$.
 
\begin{corollary}\label{Corollary:TwoPointsOddEven}
Let $f$ be an entire function satisfying \eqref{eq:maingrowthcondition} for which $f^{(2n+1)}(s_0)\in\Z$ and $f^{(2n)}( s_1)\in\Z$ for all sufficiently large $n$. Then the two sets 
$$
\{n\ge 0\; |\; f^{(2n+1)}(s_0)\not=0\} \quad \text{and}\quad 
\{n\ge 0\; |\; f^{(2n)}( s_1)\not=0\}
$$
are finite.
If the exponential type of $f$ satisfies $\tau(f)<\frac{\pi}{2| s_1-s_0|}$, then $f$ is a polynomial. 
\end{corollary}

The assumption \eqref{eq:maingrowthcondition} in Corollary \ref{Corollary:TwoPointsOddEven} is essentially optimal: 

\begin{theorem}\label{Th:TwoPointsOddEvenExistence}
Let $s_0, s_1$ be two distinct complex numbers satisfying
\begin{equation}\label{equation:majoration|s0-s1|OddEven}
|s_1-s_0|<\log(2+\sqrt 3)=1.316\, 957\,8\cdots . 
\end{equation}
There exist a constant $\gamma'$ and an uncountable set of transcendental entire functions $f$ satisfying $f^{(2n+1)}(s_0)=0$ and $f^{(2n)}( s_1)\in\Z$ for all $n\ge 0$, such that the set of $n\ge 0$ with $f^{(2n)}(s_1)\not=0$ is infinite and such that 
\begin{equation}\label{eq:maingrowthconditiongammaprime}
\limsup_{r\to\infty}\rme^{-r}\sqrt r |f|_r\le \gamma'.
\end{equation}
\end{theorem}
\subsection{Sequence of derivatives}\label{S:Sequences}

We propose some generalisations of Corollary \ref{Corollary:TwoPointsOddEven}, where we assume that for each sufficiently large integer $n$, one at least of the two numbers $f^{(n)}(s_0)$, $f^{(n)}(s_1)$ is in $\Z$. 

We start with the case of a periodic sequence. Let $m\ge 2$ be a positive integer. Let $\sigma_0,\sigma_1,\dots,\sigma_{m-1}$ be complex numbers, not necessarily distinct: we will be interested with the case where they all belong to a set with two elements, but the next result is not restricted to two points. Set $\zeta=\rme^{2i\pi/m}$ and denote by $\tau$ the smallest modulus of a zero of the function $\Delta(t)$, where $\Delta(t)$ is the determinant of the $m\times m$ matrix
$$
\Bigl(
\zeta^{k\ell}\rme^{\zeta^kt\sigma_\ell} 
\Bigr)_{0\le k,\ell\le m-1}=
\begin{pmatrix}
\rme^{t\sigma_0} & \rme^{t\sigma_1}& \rme^{t\sigma_2} & \cdots & \rme^{t\sigma_{m-1}} 
\\
\rme^{\zeta t \sigma_0} &\zeta\rme^{\zeta t \sigma_1} &\zeta^2 \rme^{\zeta t \sigma_2} & \cdots & \zeta^{m-1}\rme^{\zeta t\sigma_{m-1}} 
\\
\rme^{\zeta^2 t \sigma_0} &\zeta^2\rme^{\zeta^2 t \sigma_1} &\zeta^4 \rme^{\zeta^2 t \sigma_2} & \cdots & \zeta^{2(m-1)}\rme^{\zeta^2 t\sigma_{m-1}} 
\\
 \vdots& \vdots&\vdots&\ddots&\vdots 
 \\
\rme^{\zeta^{m-1} t\sigma_0} &\zeta^{m-1}\rme^{\zeta^{m-1}t\sigma_1} &\zeta^{2(m-1)} \rme^{\zeta^{m-1}t\sigma_2} & \cdots & \zeta^{(m-1)^2}\rme^{\zeta^{m-1} t\sigma_{m-1}} 
\end{pmatrix}.
$$

\begin{theorem}\label{Th:GontcharoffMacintyre}
Let $f$ be a transcendental entire function of exponential type $< \tau$ satisfying \eqref{eq:maingrowthcondition}. 
Assume that for each sufficiently large $n$, we have 
$$
f^{(mn+j)}(\sigma_j)\in\Z \text{ for $j=0,1,\dots,m-1$}.
$$
Then $f$ is a polynomial. 
 \end{theorem}
 
 This result is optimal: 
 
 \begin{proposition}\label{Prop:GontcharoffMacintyre}
 {\rm (a)} Let $\alpha$ be a zero of $\Delta(t)$. There exists $c_0,c_1,\dots,c_{m-1}$ in $\C$, not all zero, such that the function 
 $$
 f(z)=c_0\rme^{\alpha z}+c_1\rme^{\zeta\alpha z}+\cdots+c_{m-1}\rme^{\zeta^{m-1}\alpha z}
 $$
 satisfies 
 $$
f^{(mn+j)}(\sigma_j) =0 \text{ for $j=0,1,\dots,m-1$ and $n\ge 0$}.
$$
{\rm (b)} Assume $\tau>1$. Given $a_0,a_1,\dots,a_{m-1}$ in $\C$, there exists a unique entire function of exponential type $\le 1$ satisfying 
 $$
f^{(mn+j)}(\sigma_j) = a_j \text{ for $j=0,1,\dots,m-1$ and $n\ge 0$}.
$$
 \end{proposition}
 
 The function given by (a) is a transcendental entire function of exponential type $|\alpha|$. If $(a_0,a_1,\dots,a_{m-1})\not=(0,0,\dots,0)$, the function $f$ given by (b) is a transcendental entire function of exponential type $1$. Notice that $f$ does not satisfy the assumption \eqref{eq:maingrowthcondition} of Theorem \ref{Th:GontcharoffMacintyre}. 
 
 Here is a corollary of Theorem \ref{Th:GontcharoffMacintyre}. We fix again an integer $m\ge 2$ and we denote by $\tau_m$ the smallest modulus of a zero of the function 
$$
1+\frac{t^m}{m!}+\frac{t^{2m}}{(2m)!}+\cdots+\frac{t^{nm}}{(nm)!}+\cdots
$$
Since $\tau_2=\pi/2$, Corollary \ref{Corollary:TwoPointsOddEven} is the case $m=2$ of the next result. 
 
\begin{corollary}\label{Cor:GontcharoffMacintyre}
Let $s_0$ and $s_1$ be two complex numbers. 
Let $f$ be a transcendental entire functions satisfying \eqref{eq:maingrowthcondition}. 
Assume that the exponential type $\tau(f)$ of $f$ satisfies 
$$
\tau(f)< \frac {\tau_m} {| s_1-s_0|}\cdotp
$$
Assume further that for each sufficiently large $n$, we have 
$$
f^{(n)}(s_0)\in\Z \text{ for $n\not \equiv 0 \bmod m$ }
\text{ and }\;
f^{(n)}(s_1)\in\Z \text{ for $n \equiv 0 \bmod m$}. 
$$ 
Then $f$ is a polynomial. 
 \end{corollary}
 
 The case $s_0=s_1=0$ is nothing else than Proposition \ref{Prop:Polya} below due to \cite{zbMATH02601316}. 
 
Corollary \ref{Cor:GontcharoffMacintyre} is sharp: from part (a) of Proposition \ref{Prop:GontcharoffMacintyre} it follows that there exists a transcendental entire function $f$ of type $\tau_m/|s_1-s_0|$ satisfying 
$$
f^{(n)}(s_0) =0 \text{ for $n\equiv 0 \bmod m$ }
\text{ and }\;
f^{(n)}(s_1)=0 \text{ for $n\not \equiv 0 \bmod m$}. 
$$ 
Also, from part (b) of Proposition \ref{Prop:GontcharoffMacintyre} it follows that if $\tau_m>|s_1-s_0|$, given $a_0,a_1,\dots, a_{m-1}$ in $\C$, not all of which are zero, there exists a unique entire function $f$ of exponential type $1$ satisfying, for all $n\ge 0$, 
$$
f^{(n)}(s_0)=a_j \text{ for $n \equiv j \bmod m$ and $1\le j\le m-1$}
\text{ and }\;
f^{(n)}(s_1)=a_0 \text{ for $n \equiv 0 \bmod m$}. 
$$ 
This function is transcendental of exponential type $1$, unless $a_0=a_1=\cdots=a_{m-1}=0$ in which case it is $0$.

The next and last result deals with a situation more general than the case of two points in Theorem \ref{Th:GontcharoffMacintyre}, since no periodicity is assumed, and we assume only that one at least of the three numbers $ f^{(n)}(s_0)$, $f^{(n)}(s_1)$, $f^{(n)}(s_0) f^{(n)}(s_1) $ is in $\Z$. The assumption on the type in Theorem \ref{Th:2PointsAlln} may not be optimal. 
 
\begin{theorem}\label{Th:2PointsAlln}
Let $s_0, s_1$ be two distinct complex numbers. 
Let $f$ is an entire function of exponential type $\tau(f)$ satisfying \eqref{eq:maingrowthcondition}. Assume 
$$
\tau(f)< \frac 1 {| s_1-s_0|}\cdotp
$$
Assume that, for all sufficiently large $n$, one at least of the three numbers
$$
 f^{(n)}(s_0), f^{(n)}(s_1), f^{(n)}(s_0) f^{(n)}(s_1) 
$$
 is in $\Z$. 
Then $f$ is a polynomial. 
\end{theorem}
 
\section{On a result of P\'olya}\label{S:GeneralizationPolya}

Recall that a Hurwitz function is an entire function satisfying $f^{(n)}(0)\in\Z$ for all $n\ge 0$. 
Here is one of the earliest results on Hurwitz functions 
\cite{zbMATH02601316}. 

\begin{proposition}\label{Prop:Polya}
A transcendental Hurwitz function $f$ satisfies
$$
\limsup_{r\to\infty}\rme^{-r}\sqrt r |f|_r\ge \frac{1}{\sqrt{2\pi}}\cdotp
$$
 \end{proposition}

The noncountable set of entire functions 
\begin{equation}\label{Eq:PolyaExample}
f(z)=\sum_{n=0}^\infty e_n \frac{z^{2^n}}{2^n!}
\quad\text{ for which }\quad
\limsup_{r\to\infty}\rme^{-r}\sqrt r |f|_r = \frac{1}{\sqrt{2\pi}},
\end{equation}
where $e_n\in\{-1,1\}$, 
shows that Proposition \ref{Prop:Polya} is optimal. This does not mean that it is the final word. On the one hand, 
\cite[\S~2 Corollary 1]{MR0159945} 
and
\cite[\S~3 Corollary]{MR0192030} 
have proved more precise results, including the following :
\begin{quote}
 For every $\epsilon>0$, there exists a transcendental Hurwitz function with 
$$
\limsup_{r\to\infty} \sqrt{2\pi r} \, \rme^{-r}\left(1+\frac{1+\epsilon}{24 r}\right)^{-1} |f|_r <1,
$$
while every Hurwitz function for which 
$$
\limsup_{r\to\infty} \sqrt{2\pi r}\, \rme^{-r}\left(1+\frac{1-\epsilon}{24 r}\right)^{-1}  |f|_r \le 1
$$
is a polynomial. 
\end{quote}

On the other hand, our Corollary \ref{Corollary:Polya} below extends the range of validity of Proposition \ref{Prop:Polya}.

\begin{proposition}\label{Proposition:Polya}
Let $f$ be an entire function and let $A\ge 0$. Assume 
\begin{equation}\label{eq:maingrowthconditionA}
\limsup_{r\to\infty}\rme^{-r}\sqrt r |f|_r< \frac{ \rme^{-A}}{\sqrt{2\pi} }\cdotp
\end{equation}
Then there exists $n_0>0$ such that, for $n\ge n_0$ and for all $z\in\C$ in the disc $|z|\le A$, we have 
$$
 |f^{(n)}(z)|<1.
$$ 
\end{proposition}

\noindent

\begin{remark}{\rm 
When $A=0$, P\'olya's example \eqref{Eq:PolyaExample}
shows that the upper bound in the assumption of Proposition \ref{Proposition:Polya} is optimal.
}
\end{remark}

For the proof of Proposition \ref{Proposition:Polya}, we will use Cauchy's inequalities for an entire function $f$:
\begin{equation}\label{Equation:CauchyInequality}
\frac{|f^{(n)}(z_0)|}{n!}r^{n} \le |f|_{r+|z_0|},
\end{equation}
which are valid for all $z_0\in\C$, $n\ge 0$ and $r>0$. We will also use Stirling's Formula:
\begin{equation}\label{Equation:Stirling}
N^N\rme^{-N}\sqrt{2\pi N}
< N! <
N^N \rme^{-N} \sqrt{2\pi N} \rme^{1/(12N)}, 	
\end{equation}
which is valid for all $N\ge 1$.

\begin{proof}[Proof of Proposition \ref{Proposition:Polya}] 
Let $\epsilon>0$. By assumption, for $n$ sufficiently large, we have 
$$
|f|_n< (1-\epsilon) \frac{ \rme^{n-A}}{\sqrt{2\pi n} }\cdotp
$$
We use Cauchy's inequalities \eqref{Equation:CauchyInequality} with $r=n-A$: for $|z|\le A$, we have 
$$
|f^{(n)}(z)|\le \frac{n!}{(n-A)^n}|f|_n.
$$
Hence \eqref{Equation:Stirling} yields
$$
|f^{(n)}(z)|\le (1-\epsilon) 
 \rme^{-A+1/(12n)} 
 \left( 1-\frac A n \right)^{-n}.
$$
For $n$ sufficiently large the right hand side is $<1$.
\end{proof}

We deduce the following refinement of Proposition \ref{Prop:Polya}: 

\begin{corollary}\label{Corollary:Polya}
Let $f$ be an entire function. Let $A\ge 0$. 
Assume 
\eqref{eq:maingrowthconditionA}. 
Then the set
$$
\bigl\{(n,z_0)\in\N\times \C \; \mid \; |z_0|\le A, \; f^{(n)}(z_0)\in\Z\setminus\{0\} \bigr\}
$$
is finite. 
\end{corollary}
 
\section{Lidstone polynomials}\label{S:Lidstone}

The theory of Lidstone polynomials and series has a long and rich history. We recall the definition and the basic results which we will need. 

\subsection{Definition and properties}\label{SS:Lidstone}

We denote by $\delta_{ij}$ the Kronecker symbol: 
$$
\delta_{ij}=\begin{cases}
1 & \text{ if $i=j$},
\\
0 & \text{ if $i\neq j$}.
\end{cases}
$$
By induction on $n$, one defines a sequence of polynomials $(\Lambda_n)_{n\ge 0}$ in $\Q[z]$
 by the conditions $ \Lambda_0(z)=z$ and 
 $$ 
\Lambda_n''=\Lambda_{n-1},\qquad \Lambda_n(0)=\Lambda_n(1)=0
\quad\hbox{for all $n\ge 1$}.
$$ 
For $n\ge 0$, the polynomial $\Lambda_n$ is odd, it has degree $2n+1$ and leading term $\frac{1}{(2n+1)!}z^{2n+1}$.
From the definition one deduces 
$$ 
 \Lambda_n^{(2k)}(0)=0 \hbox{ and } \Lambda_n^{(2k)}(1)=\delta_{k,n} \hbox{ for all $n\ge 0$ and $k\ge 0$}.
$$ 

This definition goes back to 
 \cite{zbMATH02567100}. 
See also
 \cite{MR1501639}, 
\cite{zbMATH02543645}, 
\cite[\S~9]{zbMATH02532117}, 
\cite{zbMATH03021531}, 
\cite[\S~9]{MR0029985}, 
\cite{MR0072947}, 
\cite[Chap.~I \S~4]{MR0162914}, 
\cite{MR2303366}, 
\cite[\S~1]{MR3849168}. 
A consequence of the definition is that any polynomial $f\in\C[z]$ has a finite expansion 
\begin{equation}\label{Equation:LidstoneExpansionPolynomials}
f(z)=\sum_{n=0}^\infty \left(
f^{(2n)}(0)\Lambda_n(1-z)+
f^{(2n)}(1)\Lambda_n(z)
\right)
\end{equation} 
with only finitely many nonzero terms in the series. 
 
Applying \eqref{Equation:LidstoneExpansionPolynomials} 
to the polynomial $ 
 z^{2n+1} $
yields the following recurrence formula \cite[Th.~2]{MR2303366}: 
for $n\ge 0$,
\begin{equation} \label{eq:Costabile}
 \Lambda_n(z) =\frac 1 {(2n+1)!} z^{2n+1} -\sum_{h=0}^{n-1}\frac 1 {(2n-2h+1)!} \Lambda_{h}(z).
 \end{equation}
For instance, 
$$
\Lambda_0(z)=z, \quad \Lambda_1(z)=\frac{1}{6} (z^3-z)
$$ 
and \cite[\S~6 p.~18]{zbMATH02567100} 
$$
 \Lambda_2(z)=\frac{1}{120} z^5-\frac{1}{36} z^3+\frac{7}{360} z= \frac{1}{360}z(z^2-1)(3z^2-7).
 $$ 
It follows from \eqref{Equation:LidstoneExpansionPolynomials} that for $n\ge 0$, a basis of the $\Q$--space 
of polynomials in $\Q[z]$ of degree $\le 2n+1$ is given by the $2n+2$ polynomials 
$$
\Lambda_0(z),\Lambda_1(z),
\dots, \Lambda_n(z), \quad
\Lambda_0(1-z),\Lambda_1(1-z),
\dots, \Lambda_n(1-z).
$$ 
Another consequence of \eqref{Equation:LidstoneExpansionPolynomials} is
$$ 
\frac{z^{2n}}{(2n)!}= \Lambda_n(1-z)+
\sum_{h=0}^n\frac 1 {(2n-2h)!} \Lambda_{h}(z) 
$$ 
for $n\ge 0$. 
 
 Lidstone expansion formula \eqref{Equation:LidstoneExpansionPolynomials} for polynomials extends to entire functions of finite exponential type
 --- see 
\cite[Th.~2]{zbMATH02543645}, 
 \cite[Th.~1]{MR1501639}, 
\cite[Th.~1]{zbMATH03021531}, 
\cite[Theorem p.~795]{MR0072947}, 
\cite[Th.~4.6]{MR0162914}. 
If $f$ has exponential type $<\pi$, then \eqref{Equation:LidstoneExpansionPolynomials} holds for $f$, the series being uniformly convergent on any compact of $\C$. 
Therefore, if an entire function $f$ has exponential type $<\pi$ and satisfies $f^{(2n)}(0)=f^{(2n)}(1)=0$ for all sufficiently large $n$, then $f$ is a polynomial. The following result deals with entire function $f$ of any finite exponential type.

\begin{proposition}\label{Proposition:LidstoneSine} 
Let $f$ be an entire function of finite exponential type $\tau(f)$ satisfying $f^{(2n)}(0)=f^{(2n)}(1)=0$ for all $n\ge 0$. Then there exist complex numbers $c_1,\dots,c_L$ with $L\le \tau(f)/\pi$ such that
$$ 
f(z)= \sum_{\ell=1}^L c_\ell \sin(\ell\pi z).
$$ 
\end{proposition}

Let $t\in\C$, $t\not\in i\pi\Z$. The entire function 
$$
f(z)=\frac{\sinh(zt)}{\sinh (t)}=\frac{\rme^{zt}-\rme^{-zt}}{\rme^t-\rme^{-t}}
$$
satisfies 
$$
f''=t^2f,\quad f(0)=0,\quad f(1)=1,
$$
hence $f^{(2n)}(0)=0$ and $f^{(2n)}(1)=t^{2n}$ for all $n\ge 0$.
Applying Proposition \ref{Proposition:LidstoneSine} yields the following expansion, valid for $0<|t|<\pi$ and $z\in\C$:
\begin{equation}\label{Equation:shcosech} 
\frac{\sinh(zt)}{\sinh (t)}=\sum_{n=0}^\infty t^{2n}\Lambda_n(z).
\end{equation}

Using Cauchy's residue Theorem with \eqref{Equation:shcosech}, we deduce the integral formula
\cite[p.~454--455]{zbMATH02543645}: 
\begin{align}\notag
\Lambda_n(z)
=(-1)^n\frac{2}{\pi^{2n+1} } 
\sum_{s=1}^{S} \frac{(-1)^s}{s^{2n+1}} 
&\sin \bigl( s\pi z\bigr)
\\
\notag
&
+\frac{1}{2\pi i} \int_{|t|=(2S+1)\pi/2} t^{-2n-1}\frac
{\sinh(zt)}{\sinh(t)}
\rmd t
\end{align}
for $S=1,2,\dots$ and $z\in\C$.
In particular, with $S=1$ we have
\begin{equation}\label{equation:IntegralFormula}
\Lambda_n(z)=(-1)^n\frac{2}{\pi^{2n+1} } \sin (\pi z)
+
\frac{1}{2\pi i} \int_{|t|=3\pi/2} t^{-2n-1}\frac
{\sinh(zt)}{\sinh(t)}
\rmd t.
\end{equation}

\subsection{Replacing $0$ and $1$ with $s_0$ and $s_1$}\label{SS:LidstoneTilde}

Let $s_0$ and $s_1$ be two distinct complex numbers. Define, for $n\ge 0$,
 $$
 \Lambdatilde_n(z)= (s_1-s_0)^{2n} \Lambda_n\left( \frac{z}{s_1-s_0}\right).
 $$
This sequence of polynomials is also defined by induction by
 $$
 \Lambdatilde_0(z)=\frac z {s_1-s_0} 
 $$
 and, for $n\ge 1$,
 $$
 \Lambdatilde_n''=\Lambdatilde_{n-1}, \quad \Lambdatilde_n(0)=\Lambdatilde_n(s_1-s_0)=0.
 $$
Hence 
 $$
 \Lambdatilde_n^{(2k)}(0)=0 \hbox{ and } \Lambdatilde_n^{(2k)}(s_1-s_0)=\delta_{k,n} \hbox{ for all $n\ge 0$ and $k\ge 0$}.
 $$
It follows that any polynomial $f\in\C[z]$ has an expansion
$$
f(z)=\sum_{n=0}^\infty 
\left( f^{(2n)}(s_1) \Lambdatilde_n(z-s_0)
-
f^{(2n)}(s_0) \Lambdatilde_n(z-s_1)\right),
$$
with only finitely many nonzero terms in the series. 

From \eqref{eq:Costabile} we deduce
\begin{equation} \label{eq:Costabiletilde}
 \Lambdatilde_n(z) =\frac {z^{2n+1}} {(s_1-s_0)(2n+1)!} -\sum_{h=0}^{n-1}\frac {(s_1-s_0)^{2n-2h} }{(2n-2h+1)!} \Lambdatilde_{h}(z).
\end{equation}
 
We will use the following elementary auxiliary lemma. 

\begin{lemma}\label{Lemma:Stirling}
There exists an absolute constant $r_0>0$ such that, for any $r\ge r_0$ and any $t$ in the interval $0<t\le r$, we have 
$$
\frac{r}{t}(1+\log t)+\frac{1}{2}\log t < r+\frac{1}{4r}\cdotp
$$
\end{lemma}

\begin{proof}
Notice first that the result is true for $0<t\le 1$ and $\sqrt r\le t\le r$.

Let $r$ be a sufficiently large positive real number. 
Define, for $t>0$
$$
f(t)=\frac{r}{t}(1+\log t)+\frac{1}{2}\log t.
$$
The derivative $f'$ of $f$ is 
$$
f'(t)=\frac{1}{2t^2}(t-2r\log t)
$$ 
and $f'(t)$ has two positive zeroes $1<t_1<t_2$, where $t_1$ is close to $1$ while $t_2$ is close to $2r\log r$ when $r$ is large. 
Since $f(e^r)<r=f(1)<f(t_1)$, in the interval $0<t\le \rme^r$, the function $f$ has its maximum at $t_1$ with $t_1=2r\log t_1$,
$$
t_1=1+\frac{1}{2r}+\frac{3}{8r^2}+\frac{1}{3r^3}+O(1/r^4)
$$
and 
$$
\log t_1= \frac{1}{2r}+\frac{1}{4r^2}+\frac{3}{16r^3}+ O(1/r^4)
$$
for $r\to\infty$.
The maximum is 
$$
f(t_1)=
\frac{r}{t_1}+\frac{1}{2}+\frac{t_1}{4r}
$$
and we have
$$
\frac{r}{ t_1}=\frac{1}{2\log t_1}=r-\frac{1}{2}-\frac{1}{8r} +O(1/r^2),
$$
so that
$$
f(t_1)= r+\frac{1}{8r}+O(1/r^2)<r+\frac{1}{4r}
$$
for sufficiently large $r$. 
\end{proof}

Setting $t=r/N$ and using the left hand side of Stirling's Formula \eqref{Equation:Stirling}, we deduce from Lemma \ref{Lemma:Stirling}:

\begin{corollary}\label{Cor:Stirling}
For sufficiently large $r$ and 
for all $N\ge 1$, we have 
$$
\frac{r^N}{N!} \le \frac{ \rme^{r+(1/4r)}}{\sqrt {2\pi r}}\cdotp
$$
\end{corollary} 
 
 Corollary \ref{Cor:Stirling} will be used in the proof of part (ii) of the following result. 
 
\begin{lemma}\label{Lemma:upperbound|Lambdatilde|} 
Let $s_0$ and $s_1$ be two distinct complex numbers. 
There exist positive contants $\gamma_1$, $\gamma_2$ and $\gamma_3$ such that the following holds. 
\\
(i) For $r\ge 0$ and $n\ge 0$, we have
 $$
| \Lambdatilde_n|_r\le \gamma_1 \frac{ |s_1-s_0|^{2n}} {(2n+1)!}\max
\left\{ \frac r {|s_1-s_0|}, 2n+1\right \}^{2n+1}.
 $$ 
 (ii) Assume \eqref{equation:majoration|s0-s1|even}. Then, 
for sufficiently large $r$, we have, for all $n\ge 0$,
$$
 |\Lambdatilde_n|_r \le \gamma_2 \frac {\rme^{r+1/(4r)}}{\sqrt {2\pi r}}\cdotp
$$
(iii) For $r\ge 0$ and $n\ge 0$, 
$$
|\Lambdatilde_n|_r\le \gamma_3 \left(\frac{|s_1-s_0|}{\pi}\right)^{2n} \rme^{\frac {3\pi r}{2|s_1-s_0|}}.
$$
 \end{lemma}
 
 \begin{proof}
 $\quad$\null
 \\
 (i) Let $(\kappa_0,\kappa_1,\kappa_2,\dots)$ be a sequence of positive numbers satisfying $ \kappa_0\ge 1$ and, for $n\ge 1$,
$$
\kappa_n\ge 1+ \sum_{h=0}^{n-1} \frac{\kappa_h}{(2n-2h+1)!} \cdotp
$$
By induction we prove the estimate, for $z\in\C$, 
\begin{equation}\label{equation:c}
| \Lambdatilde_n(z)|\le \kappa_n \frac{ |s_1-s_0|^{2n}} {(2n+1)!}\max
\left\{ \frac {|z|} {|s_1-s_0|}, 2n+1\right \}^{2n+1}.
\end{equation}
Formula \eqref{equation:c} is true for $n=0$. Assume that, for some $n\ge 1$, \eqref{equation:c} is true for $n$ replaced with $h$ for all $h=0,1,\dots,n-1$. Then for $0\le h\le n-1$ we have 
$$
|\Lambdatilde_h(z)|
\le
\kappa_h
 \frac{ |s_1-s_0|^{2h}} {(2h+1)!}\max
\left\{ \frac {|z|} {|s_1-s_0|}, 2n+1\right \}^{2h+1}.
$$
We use the upper bound 
$$
\frac{(2n+1)!}{(2h+1)!}\le (2n+1)^{2n-2h}.
$$
We deduce, for $0\le h\le n-1$,
$$
|s_1-s_0|^{2n-2h}
|\Lambdatilde_h(z)|
\le
\kappa_{h}
 \frac{ |s_1-s_0|^{2n}} {(2n+1)!}\max
\left\{ \frac {|z|} {|s_1-s_0|}, 2n+1\right \}^{2n+1}.
$$
Now \eqref{eq:Costabiletilde} implies
$$
| \Lambdatilde_n(z)|\le
\left(
1+
\sum_{h=0}^{n-1} \frac{ \kappa_h }{(2n-2h+1)!} \right)
 \frac{ |s_1-s_0|^{2n}} {(2n+1)!}\max
\left\{ \frac {|z|} {|s_1-s_0|}, 2n+1\right \}^{2n+1},
$$
which proves \eqref{equation:c}.

We deduce part (i) of Lemma \ref{Lemma:upperbound|Lambdatilde|} by taking for the sequence $(\kappa_h)_{h\ge 0}$ a constant sequence $\kappa_h=\gamma_1$ with
$$
\gamma_1=1+\gamma_1\sum_{\ell\ge 1} \frac{1}{(2\ell+1)!}\cdotp
$$
This proves (i) with the explicit value
$$
\gamma_1=\frac{2}{ 4-\rme +\rme^{-1}}=1.212\, 136\, 8\dots
$$

\medskip\noindent
(ii)
Fix $r$ sufficiently large. 
Let $(\kappatilde_n)_{n\ge 0}$ be another sequence of positive real numbers satisfying $ \kappatilde_0\ge 1/|s_1-s_0|$ and, for $n\ge 1$,
\begin{equation}\label{eq:kappatilde}
\kappatilde_n\ge \frac 1 {|s_1-s_0|}+ \sum_{h=0}^{n-1} \kappatilde_h \frac{|s_1-s_0|^{2n-2h}}{(2n-2h+1)!} \cdotp
\end{equation}
We prove the estimate 
\begin{equation}\label{equation:gamma}
| \Lambdatilde_n|_r\le \kappatilde_n \frac{ \rme^{r+(1/4r)}}{\sqrt {2\pi r}} \cdotp
\end{equation}
This is true for $n=0$, since $r$ is sufficiently large.
Assume that it is true for all $h$ with $0\le h<n$ for some $n\ge 1$. Using the induction hypothesis with \eqref{eq:Costabiletilde}, we obtain
 $$
 | \Lambdatilde_n|_r\le \frac{r^{2n+1}}{|s_1-s_0| (2n+1)!} +
 \frac{ \rme^{r+(1/4r)}}{\sqrt {2\pi r}} 
 \sum_{h=0}^{n-1}\kappatilde_h \frac{|s_1-s_0|^{2n-2h}}{(2n-2h+1)!}\cdotp
$$
Now \eqref{equation:gamma} follows from \eqref{eq:kappatilde} and Corollary \ref{Cor:Stirling}. 
We take for the sequence $(\kappatilde_h)_{h\ge 0}$ a constant sequence $\kappatilde_h=\gamma_2$ with
$$
\gamma_2= \frac 1 {|s_1-s_0|}+\gamma_2\sum_{\ell\ge 1} \frac{|s_1-s_0|^{2\ell}}{(2\ell+1)!}\cdotp
$$
Since \eqref{equation:majoration|s0-s1|even} can be written
$$
4|s_1-s_0|-\rme^{|s_1-s_0|}+\rme^{-|s_1-s_0|}>0, 
$$
we deduce part (ii) of Lemma \ref{Lemma:upperbound|Lambdatilde|} with 
$$
\gamma_2=
\frac{2}{ 4|s_1-s_0|-\rme^{|s_1-s_0|}+\rme^{-|s_1-s_0|}}\cdotp
$$ 

\medskip\noindent
(iii) From \eqref{equation:IntegralFormula} we deduce
$$
|\Lambda_n|_r\le 
\rme^{(3\pi/2) r} \pi^{-2n}
\left(\frac{2}{\pi} \rme^{-\pi r/2}+
\frac{2^{2n+1}}{3^{2n}} \sup_{|t|=3\pi/2} \frac{1}{ |\rme^t-\rme^{-t}|
}
\right).
$$ 
The proof of Lemma \ref{Lemma:upperbound|Lambdatilde|} is complete.
 \end{proof}
 
From part (iii) of Lemma \ref{Lemma:upperbound|Lambdatilde|} we deduce the following. 

\begin{corollary}\label{Cor:uperboundLidstoneWhittaker}
Assume $|s_1-s_0|<\pi$. There exists a constant $\gamma_4>0$ such that, for $r$ sufficiently large, 
$$
\sum_{n\ge \gamma_4 r}|\Lambdatilde_n|_r<1.
$$
\end{corollary} 

From \eqref{Equation:shcosech} it follows that the assumption $|s_1-s_0|<\pi$ cannot be relaxed. 

\begin{proof}[Proof of Corollary \ref{Cor:uperboundLidstoneWhittaker}] 
Let $N$ be a positive integer. From part (iii) of Lemma \ref{Lemma:upperbound|Lambdatilde|} we deduce 
\begin{align}\notag
\sum_{n\ge N}|\Lambdatilde_n|_r&\le \gamma_3 
\rme^{\frac {3\pi r}{2|s_1-s_0|}}\sum_{n\ge N} \left(\frac{|s_1-s_0|}{\pi}\right)^{2n} 
\\
\notag
&=
\frac{\gamma_3\pi^2}{\pi^2-|s_1-s_0|^2}\rme^{\frac {3\pi r}{2|s_1-s_0|}}\left(\frac{|s_1-s_0|}{\pi}\right)^{2N}\cdotp
\end{align}
The right hand side is $<1$ as soon as 
$$
\frac {3\pi r}{2|s_1-s_0|} +\log \frac{\gamma_3\pi^2}{\pi^2-|s_1-s_0|^2}<2N\log \frac{\pi}{|s_1-s_0|},
$$
and this is true for $r$ sufficiently large and $N\ge \gamma_4r$, provided that 
$$
\gamma_4>\frac{3\pi}{4|s_1-s_0|(\log \pi - \log |s_1-s_0|)}\cdotp
$$ 
 \end{proof} 
 
\section{Derivatives of even order at two points. 
}\label{S:Even}
 
\subsection{Proof of Theorem \ref{Th:TwoPointsEven}} 

\begin{proof}[Proof of Theorem \ref{Th:TwoPointsEven}]
Let $f$ satisfy the assumptions of Theorem \ref{Th:TwoPointsEven}. Using 
Corollary \ref{Corollary:Polya}, we deduce from the assumption \eqref{eq:maingrowthcondition} that 
 the sets
$$ 
\{ n\ge 0\; \mid \; f^{(2n)}(s_0)\not=0\}
\;
\text{ and } 
\;
\{ n\ge 0\; \mid \; f^{(2n)}( s_1)\not=0\}
$$ 
are finite. Hence
$$
P(z)=
\sum_{n=0}^\infty 
\left( f^{(2n)}(s_1) \Lambdatilde_n(z-s_0)
+ 
f^{(2n)}(s_0) \Lambdatilde_n(z-s_1)\right)
$$
is a polynomial satisfying 
$$
P^{(2n)}(s_0)=f^{(2n)}(s_0)
\;
\text{ and }
\;
P^{(2n)}( s_1)=f^{(2n)}( s_1)
\;
\text{ for all } 
\;
n\ge 0.
$$
The function $\ftilde(z)=f(z)-P(z)$ has the same exponential type as $f$ and satisfies 
 $$
 \ftilde^{(2n)}(s_0)= \ftilde^{(2n)}( s_1)=0 
 \;
 \text { for all }
 \; 
 n\ge 0.
 $$ 
Set
$$
\fhat(z)=\ftilde\bigl(s_0+z( s_1-s_0)\bigr),
$$
so that 
 $$
 \fhat^{(2n)}(0)= \fhat^{(2n)}(1)=0
 \;
 \text { for all }
 \;
 n\ge 0.
 $$ 
The exponential types of $f$ and $\fhat$ are related by
$$
\tau(\fhat)=| s_1-s_0|\tau(f).
$$ 
From Proposition \ref{Proposition:LidstoneSine} we deduce that there exists complex numbers $c_1,c_2,\dots,c_L$ such that
$$
\fhat(z)=\sum_{\ell=1}^L c_\ell \sin(\ell\pi z),
$$ 
and therefore
$$
\ftilde(z)= \sum_{\ell=1}^L c_\ell \sin\left(\ell\pi\frac{z-s_0}{ s_1-s_0}\right).
$$ 
Theorem \ref{Th:TwoPointsEven} follows.
\end{proof}

\subsection{Proof of Theorem \ref{Th:TwoPointsEvenExistence}} 

\begin{proof}[Proof of Theorem \ref{Th:TwoPointsEvenExistence}] 
Assume $| s_1-s_0|<\pi$. 
From Proposition \ref{Proposition:LidstoneSine}, it follows that an entire function $f$ of exponential type $\le 1$ for which $f^{(2n)}(s_0)\in\Z$ and $f^{(2n)}( s_1)\in\Z$ for all sufficiently large $n$ is of the form 
 $$
 f(z)=\sum_{n=0}^\infty \left(
f^{(2n)}(s_0)\Lambdatilde_n(z-s_1)+
f^{(2n)}(s_1)\Lambdatilde_n(z-s_0)
\right),
$$ 
and that $f$ is not a polynomial if and only if one at least of the two sets $\{n\ge 0\;\mid\;f^{(2n)}(s_0)\not=0\}$, $\{n\ge 0\;\mid\;f^{(2n)}(s_1)\not=0\}$ is infinite. 
We construct such functions by requiring $f^{(2n)}(s_0)=0$ for all $n\ge 0$ and $f^{(2n)}(s_1)=0$ for all $n\ge 0$ outside a lacunary sequence. 
 
Define, for $k\ge 0$, $N_k=\gamma_4^{2^k-1}$, where $\gamma_4$ is the constant in Corollary \ref{Cor:uperboundLidstoneWhittaker}, so that $N_0=1$ and $N_{k+1}=\gamma_4N_k^2$.
For $n\ge 1$, let $e_n=0$ if $N_k<n<N_{k+1}$, and $e_{N_k} \in\{+1,-1\}$ for $k\ge 0$, so that there is an infinite uncountable set of such lacunary sequences $(e_n)_{n\ge 0}$. Define 
$$
f(z):=\sum_{n\ge 1} e_n\Lambdatilde_n(z-s_0).
$$ 
This function $f$ is entire of order $\le 1$, and $f^{(2n)}(s_0)=0$, $f^{(2n)}(s_1)=e_n$ for all $n\ge 0$. Since infinitely many $e_n$ are not $0$, this function $f$ is transcendental. It remains to check the upper bound for $|f|_r$. 

Let $r$ be a sufficiently large positive number. Let $k$ be the least positive integer such that $N_k>\sqrt {r+|s_0|}$. 
From part (i) of Lemma \ref{Lemma:upperbound|Lambdatilde|}, using the bound $N_{k-1}\le \sqrt {r+|s_0|}\le \sqrt {2r}$, we deduce, for sufficiently large $r$, 
\begin{align}
\notag
\sum_{n< N_k} |e_n| \; |\Lambdatilde_n|_{r+|s_0|}
& 
\le 
\sum_{n\le N_{k-1}} |\Lambdatilde_n|_{r+|s_0|}
\\
\notag
&< 
\gamma_1 \frac {N_{k-1}} { |s_1-s_0|} (2r)^{2N_{k-1}+1}
\\
\notag
&
\le \gamma_1 r^{3\sqrt r }
\\
\notag
&< \frac{\rme^r}{r}\cdotp
\end{align}
Assuming \eqref{equation:majoration|s0-s1|even}, we can use part (ii) of Lemma \ref{Lemma:upperbound|Lambdatilde|} and get
$$
 |\Lambdatilde_{N_k}|_{r+|s_0|}\le \gamma_2 \frac{\rme^{r+|s_0|+1/(4r)}}{\sqrt {2\pi r}}\cdotp
$$
Since $\gamma_4 (r+|s_0|)\le \gamma_4 N_k^2=N_{k+1}$, Corollary \ref{Cor:uperboundLidstoneWhittaker} yields
 $$
\sum_{n> N_k} |e_n| \; |\Lambdatilde_n|_{r+|s_0|}
\le \sum_{n\ge N_{k+1}} \; |\Lambdatilde_n|_{r+|s_0|}<1.
$$ 
Combining these three estimates, we conclude
$$
\limsup_{r\to\infty}\rme^{-r}\sqrt r |f|_r\le \gamma
\text{ with }
\gamma=\gamma_2 \frac{\rme^{|s_0|}}{\sqrt {2\pi}},
$$
which is an explicit version of \eqref{eq:maingrowthconditiongamma}:
$$
\gamma= \frac{\rme^{|s_0|}}{\sqrt {2\pi}} \cdot
\frac{2}{ 4|s_1-s_0|-\rme^{|s_1-s_0|}+\rme^{-|s_1-s_0|}}\cdotp
$$
\end{proof}

\section{Whittaker polynomials}\label{S:Whittaker}

\subsection{Definition and properties}\label{SS:WhittakerPolynomials}

We now consider the set $\scrS=(\{0\}\times (2\N+1))\cup (\{1\}\times 2\N)\subset\{0,1\}\times \N$: we take odd derivatives at $0$ and even derivatives at $1$. The analogs of Lidstone polynomials have been introduced by 
\cite[\S~6 p.~457--458]{zbMATH02543645}, 
and studied by
 \cite{zbMATH03021531}. 
See also 
\cite[Chap.~III \S~4]{zbMATH03416363}. 
 
Following \cite{zbMATH02543645}, 
one defines a sequence $(M_n)_{n\ge 0}$ of even polynomials by induction on $n$ with $M_0=1$, 
$$
M''_n=M_{n-1}, \quad M_n(1)=M'_n(0)=0 \text{ for all }n\ge 1.
$$
For all $n\ge 0$, the polynomial $M_n$ is even of degree $2n$ and leading term $\frac{1}{(2n)!}z^{2n}$. From the definition one deduces 
$$
M_n^{(2k+1)}(0)=0 \hbox{ and } M_n^{(2k)}(1)=\delta_{k,n} \hbox{ for all $n\ge 0$ and $k\ge 0$}.
$$
As a consequence, any polynomial $f\in\C[z]$ has an expansion
\begin{equation}\label{Equation:WhittakerExpansionPolynomials}
f(z)=\sum_{n=0}^\infty 
\left( f^{(2n)}(1) M_n(z)- 
f^{(2n+1)}(0) M'_{n+1}(1-z)\right),
\end{equation}
with only finitely many nonzero terms in the series. 

Applying \eqref{Equation:WhittakerExpansionPolynomials}
to the polynomial $z^{2n} $
yields the following recurrence formula 
\begin{equation} \label{eq:AnalogueCostabile}
M_n(z) =\frac 1 {(2n)!} z^{2n} -
\sum_{h=0}^{n-1}\frac 1 {(2n-2h)!} M_h(z).
 \end{equation}
 For instance 
$$
M_1(z)=\frac 12 (z^2-1),\quad M_2(z)=\frac 1 {24} (z^4-6z^2+5)=\frac 1 {24} (z^2-1)(z^2-5),
$$
$$
M_3(z)=\frac{1}{720} (z^6-15z^4+75z^2-61)=\frac{1}{720} (z^2-1)(z^4-14z^2+61).
$$

Here is the analog of Proposition \ref{Proposition:LidstoneSine} for Whittaker polynomials
 \cite[Th.~2]{zbMATH03021531}: 

\begin{proposition}\label{Proposition:Schoenberg-cosine}
Let $f$ be an entire function of finite exponential type $\tau(f)$ satisfying $f^{(2n+1)}(0)=f^{(2n)}(1)=0$ for all $n\ge 0$. Then there exist complex numbers $c_0,\dots,c_L$ with $(2L+1)\pi/2\le \tau(f)$ such that
$$ 
f(z)= \sum_{\ell=0}^L c_\ell \cos\left(\frac{(2\ell+1)\pi}{2} z \right).
$$
Therefore, if an entire function $f$ has exponential type $<\pi/2$ and satisfies $f^{(2n+1)}(0)=f^{(2n)}(1)=0$ for all $n\ge 0$, then $f=0$. 
 \end{proposition}
 
 In \eqref{Equation:shcosech}, we considered, for $t\in\C$, $t\not\in i\pi\Z$, the entire function $z\mapsto \frac{\sinh(zt)}{\sinh (t)}$; now we consider, for $t\in\C$, $t\not\in i\frac \pi 2+i\pi\Z$, the entire function 
$$
f(z)=\frac{\cosh(zt)}{\cosh (t)} =\frac{\rme^{zt}+\rme^{-zt}}{\rme^t+\rme^{-t}},
$$ 
which satisfies 
$$
f''=t^2f,\quad f(1)=1,\quad f'(0)=0,
$$
hence $f^{(2n)}(1)=t^{2n}$ and $f^{(2n+1)}(0)=0$ for all $n\ge 0$.
It follows from Proposition \ref{Proposition:Schoenberg-cosine} that the sequence $(M_n)_{n\ge 0}$ is also defined by the expansion
\begin{equation}\label{Equation:coshsech}
\frac{\cosh(zt)}{\cosh (t)} = \sum_{n=0}^\infty t^{2n}M_n(z) 
\end{equation}
for $|t|<\pi/2$ and $z\in\C$. 

Using Cauchy's residue Theorem, we deduce from \eqref{Equation:coshsech} the integral formula
\begin{align}\notag
M_n(z)
=(-1)^n\frac{2^{2n+2} }{\pi^{2n+1} } 
\sum_{s=0}^{S-1} \frac{(-1)^s}{(2s+1)^{2n+1}} 
&\cos \left(\frac{(2s+1)\pi}{2} z\right)
\\
\notag
&
+\frac{1}{2\pi i} \int_{|t|=S\pi} t^{-2n-1} \frac{\cosh(zt)}{\cosh (t)} 
\rmd t
\end{align}
for $S=1,2,\dots$ and $z\in\C$. In particular, with $S=1$ we obtain 
\begin{equation}\label{Eq:integraleformulaMn}
M_n(z)=(-1)^n\frac{2^{2n+2} }{\pi^{2n+1} } 
\cos (\pi z/2 )
+
\frac{1}{2\pi i} \int_{|t|=\pi} t^{-2n-1}\frac{\cosh(zt)}{\cosh (t)} 
\rmd t.
\end{equation}

\subsection{Replacing $0$ and $1$ with $s_0$ and $s_1$}

Let $s_0$ and $s_1$ be two distinct complex numbers. Define, for $n\ge 0$,
 $$
 \Mtilde_n(z)= (s_1-s_0)^{2n} M_n\left( \frac{z}{s_1-s_0}\right).
 $$
This sequence of polynomials is also defined by induction by $\Mtilde_0(z)=1$ and, for $n\ge 1$,
 $$
 \Mtilde_n''=\Mtilde_{n-1}, \quad \Mtilde'_n(0)=\Mtilde_n(s_1-s_0)=0.
 $$
Hence 
 $$
 \Mtilde_n^{(2k+1)}(0)=0 \hbox{ and } \Mtilde_n^{(2k)}(s_1-s_0)=\delta_{k,n} \hbox{ for all $n\ge 0$ and $k\ge 0$}.
 $$
It follows that any polynomial $f\in\C[z]$ has an expansion
$$
f(z)=\sum_{n=0}^\infty 
\left( f^{(2n)}(s_1) \Mtilde_n(z-s_0)
+ 
f^{(2n+1)}(s_0) \Mtilde'_{n+1}(z-s_1)\right),
$$
with only finitely many nonzero terms in the series. 

From \eqref{Equation:WhittakerExpansionPolynomials} we deduce
\begin{equation} \label{equation:CostabiletildeM}
 \Mtilde_n(z) =\frac {z^{2n}} { (2n)!} -\sum_{h=0}^{n-1}\frac {(s_1-s_0)^{2n-2h} }{(2n-2h)!} \Mtilde_{h}(z).
\end{equation}

 Here is the analog of Lemma \ref{Lemma:upperbound|Lambdatilde|} for the sequence of polynomials $\Mtilde_n$: 

\begin{lemma}\label{Lemma:upperbound|Milde|} 
Let $s_0$, $s_1$ be two distinct complex numbers. There exist positive contants $\gamma'_1$, $\gamma'_2$ and $\gamma'_3$ such that the following holds. 
\\
(i) For $r\ge 0$ and $n\ge 0$, we have
 $$
| \Mtilde_n|_r\le \gamma'_1 \frac{ |s_1-s_0|^{2n}} {(2n)!}\max
\left\{ \frac r {|s_1-s_0|}, 2n\right \}^{2n}.
 $$ 
 (ii) Assume \eqref{equation:majoration|s0-s1|OddEven}. 
Then, for sufficiently large $r$ and for all $n\ge 0$,
$$
|\Mtilde_n|_r\le \gamma'_2 \frac {\rme^{r+1/(4r)}}{\sqrt {2\pi r}}\cdotp
$$
(iii) For $r\ge 0$ and $n\ge 0$, 
$$
|\Mtilde_n|_r\le \gamma'_3 \left(\frac{2|s_1-s_0|}{\pi}\right)^{2n} \rme^{\frac {\pi r}{|s_1-s_0|}}.
$$
 \end{lemma}
 
 \begin{proof}
 $\quad$\null
 \\
 (i) Let $(\kappa'_0,\kappa'_1,\kappa'_2,\dots)$ be a sequence of positive numbers satisfying $ \kappa'_0\ge 1$ and, for $n\ge 1$,
$$
\kappa'_n\ge 1+ \sum_{h=0}^{n-1} \frac{\kappa'_h}{(2n-2h)!} \cdotp
$$
By induction we prove the estimate, for $z\in\C$, 
\begin{equation}\label{equation:cprime}
| \Mtilde_n(z)|\le \kappa'_n \frac{ |s_1-s_0|^{2n}} {(2n)!}\max
\left\{ \frac {|z|} {|s_1-s_0|}, 2n\right \}^{2n}.
\end{equation}
This is true for $n=0$ (and $z\not=0$). Assume that, for some $n\ge 1$, \eqref{equation:cprime} is true for $n$ replaced with $h$ and all $h=0,1,\dots,n-1$. Then for $0\le h\le n-1$ we have 
$$
|\Mtilde_h(z)|
\le
\kappa'_{h}
 \frac{ |s_1-s_0|^{2h}} {(2h)!}\max
\left\{ \frac {|z|} {|s_1-s_0|}, 2n\right \}^{2h}.
$$
We use the upper bound 
$$
\frac{(2n)!}{(2h)!}\le (2n)^{2n-2h}.
$$
We deduce, for $0\le h\le n-1$,
$$
|s_1-s_0|^{2n-2h}
|\Mtilde_h(z)|
\le
\kappa'_{h}
 \frac{ |s_1-s_0|^{2n}} {(2n)!}\max
\left\{ \frac {|z|} {|s_1-s_0|}, 2n\right \}^{2n}.
$$
Now \eqref{equation:CostabiletildeM} implies
$$
| \Mtilde_n(z)|\le
\left(
1+
\sum_{h=0}^{n-1} \frac{ \kappa'_h }{(2n-2h)!} \right)
 \frac{ |s_1-s_0|^{2n}} {(2n)!}\max
\left\{ \frac {|z|} {|s_1-s_0|}, 2n\right \}^{2n},
$$
which proves \eqref{equation:cprime}.

We deduce part (i) of Lemma \ref{Lemma:upperbound|Milde|} by taking for the sequence $(\kappa'_h)_{h\ge 0}$ a constant sequence $\kappa'_h=\gamma'_1$ with
$$
\gamma'_1=1+\gamma'_1\sum_{\ell\ge 1} \frac{1}{(2\ell)!}\cdotp
$$
This proves (i) with the explicit value
$$
\gamma'_1=\frac{2}{ 4-\rme -\rme^{-1}}=2.188\, 569\, 9\dots
$$

\medskip\noindent
(ii)
Fix $r$ sufficiently large. 
Let $(\kappatilde'_n)_{n\ge 0}$ be another sequence satisfying $ \kappatilde'_0> 0$ and, for $n\ge 1$,
\begin{equation}\label{eq:kappatildeprime}
\kappatilde'_n\ge 1+ \sum_{h=0}^{n-1} \kappatilde'_h \frac{|s_1-s_0|^{2n-2h}}{(2n-2h)!} \cdotp
\end{equation}
We prove the estimate 
\begin{equation}\label{equation:gammaprime}
| \Mtilde_n|_r\le \kappatilde'_n \frac{ \rme^{r+(1/4r)}}{\sqrt {2\pi r}} \cdotp
\end{equation}
This is true for $n=0$, since $r$ is sufficiently large and $\kappatilde'_0> 0$.
Assume that, for some $n\ge 1$, \eqref{equation:gammaprime} is true for all $h$ with $0\le h<n$. Using the induction hypothesis with \eqref{equation:CostabiletildeM}, we obtain
 $$
 | \Mtilde_n|_r\le \frac{r^{2n}}{ (2n)!} +
 \frac{ \rme^{r+(1/4r)}}{\sqrt {2\pi r}} 
 \sum_{h=0}^{n-1} \kappatilde'_h \frac{|s_1-s_0|^{2n-2h}}{(2n-2h)!}\cdotp
$$
Now \eqref{equation:gammaprime} follows from \eqref{eq:kappatildeprime} and Corollary \ref{Cor:Stirling}. 
We take for the sequence $(\kappatilde'_h)_{h\ge 0}$ a constant sequence $\kappatilde'_h=\gamma'_2$ with
$$
\gamma'_2= 1+\gamma'_2\sum_{\ell\ge 1} \frac{|s_1-s_0|^{2\ell}}{(2\ell)!}\cdotp
$$
Since \eqref{equation:majoration|s0-s1|OddEven} can be written
$$
\rme^{|s_1-s_0|}+ \rme^{-|s_1-s_0|}<4,
$$
this implies part (ii) of Lemma \ref{Lemma:upperbound|Milde|} with
$$
\gamma'_2=
\frac{2}{ 4-\rme^{|s_1-s_0|}-\rme^{-|s_1-s_0|}}, 
$$
provided that $r$ is sufficiently large. 

\medskip\noindent
(iii) 
From the integral formula \eqref{Eq:integraleformulaMn} one deduces the upper bound: 
$$
|M_n|_r\le 
\left(\frac{2}{\pi}\right)^{2n} \rme^{\pi r}
\left(\frac{4}{\pi}\rme^{-\pi r/2}+
 2^{-2n+1} 
\sup_{|t|=\pi} \frac{1}{ |\rme^t+\rme^{-t}|} 
\right).
$$ 
The proof of Lemma \ref{Lemma:upperbound|Milde|} is complete.
 \end{proof}
 
From part (iii) of Lemma \ref{Lemma:upperbound|Milde|} we deduce the following. 

\begin{corollary}\label{Cor:uperboundMtilde}
Assume $|s_1-s_0|<\pi/2$. There exists a constant $\gamma'_4>0$ such that, for $r$ sufficiently large, 
$$
\sum_{n\ge \gamma'_4 r}|\Mtilde_n|_r<1.
$$
\end{corollary} 

From \eqref{Equation:coshsech} it follows that the assumption $|s_1-s_0|<\pi/2$ cannot be relaxed. 

\begin{proof}[Proof of Corollary \ref{Cor:uperboundMtilde}] 
Let $N$ be a positive integer. From part (iii) of Lemma \ref{Lemma:upperbound|Milde|} we deduce 
\begin{align}\notag
\sum_{n\ge N}|\Mtilde_n|_r&\le \gamma'_3 
\rme^{\frac {\pi r}{|s_1-s_0|}}\sum_{n\ge N} \left(\frac{2|s_1-s_0|}{\pi}\right)^{2n} 
\\
\notag
&=
\frac{\gamma'_3\pi^2}{\pi^2-4|s_1-s_0|^2}\rme^{\frac {\pi r}{|s_1-s_0|}}\left(\frac{2|s_1-s_0|}{\pi}\right)^{2N}\cdotp
\end{align}
The right hand side is $<1$ as soon as 
$$
\frac {\pi r}{|s_1-s_0|} +\log \frac{\gamma'_3\pi^2}{\pi^2-4|s_1-s_0|^2}<2N\log \frac{\pi}{2|s_1-s_0|},
$$
and this is true for $r$ sufficiently large and $N\ge \gamma'_4r$, provided that 
$$
\gamma'_4>\frac{\pi}{2|s_1-s_0|(\log \pi - \log (2|s_1-s_0|))}\cdotp
$$ 
 \end{proof}
 
\section{Derivatives of odd order at one point and even at the other}\label{S:OddEven}

\subsection{Proof of Theorem \ref{Th:TwoPointsOddEven} } \label{SS:twopointsOddEven}

\begin{proof}[Proof of Theorem \ref{Th:TwoPointsOddEven}] 
Let $f$ satisfy the assumptions of Theorem \ref{Th:TwoPointsOddEven}.
Using the assumption \eqref{eq:maingrowthcondition}, we deduce from
Corollary \ref{Corollary:Polya} that the sets
$$ 
\{ n\ge 0\; \mid \; f^{(2n+1)}(s_0)\not=0\}
\;
\text{ and } 
\;
\{ n\ge 0\; \mid \; f^{(2n)}( s_1)\not=0\}
$$ 
are finite. Hence
$$
P(z)=
\sum_{n=0}^\infty 
\left( f^{(2n)}(s_1) \Mtilde_n(z-s_0)
+ 
f^{(2n+1)}(s_0) \Mtilde'_{n+1}(z-s_1)\right)
$$
is a polynomial satisfying 
$$
P^{(2n+1)}(s_0)=f^{(2n+1)}(s_0)
\;
\text{ and }
\;
P^{(2n)}( s_1)=f^{(2n)}( s_1)
\;
\text{ for all }
\;
 n\ge 0.
$$
The function $\ftilde(z)=f(z)-P(z)$ has the same exponential type as $f$ and satisfies 
 $$
 \ftilde^{(2n+1)}(s_0)= \ftilde^{(2n)}( s_1)=0
 \;
 \text { for all }
 \;
 n\ge 0.
 $$ 
Set
$$
\fhat(z)=\ftilde\bigl(s_0+z( s_1-s_0)\bigr),
$$
so that 
 $$
 \fhat^{(2n+1)}(0)= \fhat^{(2n)}(1)=0 
 \;
 \text { for all }
 \;
 n\ge 0.
 $$ 
The exponential types of $f$ and $\fhat$ are related by
$$
\tau(\fhat)=| s_1-s_0|\tau(f).
$$ 
From Proposition \ref{Proposition:Schoenberg-cosine} we deduce that there exists complex numbers $c_0,c_1,\dots,c_L$ 
with $(2L+1)\pi/2\le \tau(\fhat)$
such that
$$
\fhat(z)= \sum_{\ell=0}^L c_\ell \cos\left(\frac{(2\ell+1)\pi}{2} z \right), 
$$ 
and therefore
$$
\ftilde(z)= \sum_{\ell=0}^L c_\ell \cos\left(\frac{(2\ell+1)\pi}{2}\cdot \frac{z-s_0}{ s_1-s_0}\right).
$$ 
Theorem \ref{Th:TwoPointsOddEven} follows.
\end{proof}

\subsection{Proof of Theorem \ref{Th:TwoPointsOddEvenExistence}}

\begin{proof}[Proof of Theorem \ref{Th:TwoPointsOddEvenExistence}] 
Assume \eqref{equation:majoration|s0-s1|OddEven}. Define, for $k\ge 0$, $N_k=(\gamma'_4)^{2^k-1}$, where $\gamma'_4$ is the constant in Corollary \ref{Cor:uperboundMtilde}, so that $N_0=1$ and $N_{k+1}=\gamma'_4N_k^2$.
For $n\ge 1$, let $e_n=0$ if $N_k<n<N_{k+1}$, and $e_{N_k} \in\{+1,-1\}$ for $k\ge 0$, so that there is an infinite uncountable set of such lacunary sequences $(e_n)_{n\ge 0}$. Define 
$$
f(z):=\sum_{n\ge 1} e_n\Mtilde_n(z-s_0).
$$ 
Let us check the upper bound for $|f|_r$.

Let $r$ be a sufficiently large positive number. Let $k$ be the least positive integer such that $N_k>\sqrt {r+|s_0|}$. 
From part (i) of Lemma \ref{Lemma:upperbound|Milde|}, using the bound $N_{k-1}\le \sqrt {r+|s_0|}\le \sqrt {2r}$, we deduce, for sufficiently large $r$, 
\begin{align}
\notag
\sum_{n< N_k} |e_n| \; |\Mtilde_n|_{r+|s_0|}
& 
\le 
\sum_{n\le N_{k-1}} |\Mtilde_n|_{r+|s_0|}
\\
\notag
&< 
\gamma'_1 N_{k-1} (2r)^{2N_{k-1}}
\\
\notag
&
\le \gamma'_1 r^{3\sqrt r }
\\
\notag
&< \frac{\rme^r}{r}\cdotp
\end{align}
Assuming \eqref{equation:majoration|s0-s1|OddEven}, we can use part (ii) of Lemma \ref{Lemma:upperbound|Milde|} and get
$$
 |\Mtilde_{N_k}|_{r+|s_0|}\le \gamma'_2 \frac{\rme^{r+|s_0|+1/(4r)}}{\sqrt {2\pi r}}\cdotp
$$
Since $\gamma'_4 (r+|s_0|)< \gamma'_4 N_k^2=N_{k+1}$, Corollary \ref{Cor:uperboundMtilde} yields
 $$
\sum_{n> N_k} |e_n| \; |\Mtilde_n|_{r+|s_0|}
\le \sum_{n\ge N_{k+1}} \; |\Mtilde_n|_{r+|s_0|}<1.
$$ 
Combining these three estimates, we conclude
$$
\limsup_{r\to\infty}\rme^{-r}\sqrt r |f|_r\le \gamma'
\text{ with }
\gamma'=\gamma'_2 \frac{\rme^{|s_0|}}{\sqrt {2\pi}},
$$
which is an explicit version of \eqref{eq:maingrowthconditiongammaprime}:
$$
\gamma'= \frac{\rme^{|s_0|}}{\sqrt {2\pi}} \cdot\frac{2}{ 4-\rme^{|s_1-s_0|}-\rme^{-|s_1-s_0|}}\cdotp
$$
We deduce that $f$ has order $\le 1$ and that $f^{(2n+1)}(s_0)=0$, $f^{(2n)}(s_1)=e_n$ for all $n\ge 0$.. 
 \end{proof}
 
\section{Sequence of derivatives} 
\label{S:SequencesDerivatives}

 \subsection{Proofs of Theorem \ref{Th:GontcharoffMacintyre}
 and Proposition \ref{Prop:GontcharoffMacintyre}}\label{SS:GontcharoffMacintyre}

 The proof of Theorem \ref{Th:GontcharoffMacintyre} 
 relies on the following result of 
 \cite[Chapter IV, \S~9]{zbMATH02563461} 
and 
 \cite[\S~4]{MR0059366}. 
 See also
 \cite[Chap.~III]{zbMATH02532117} 
and
\cite[Chap.~3]{zbMATH03416363}. 

\begin{proposition}\label{Proposition:GontcharoffMacintyre}
Let $\sigma_0,\sigma_1,\dots,\sigma_{m-1}$ be complex numbers. 
If $f$ is an entire function of exponential type $<\tau$ satisfying 
$$
f^{(mn+j)}(\sigma_j)=0 \text{ for } j=0,\dots,m-1 \text{ and all sufficiently large } n,
$$
then $f$ is a polynomial
\end{proposition}

\begin{proof}[Proof of Theorem \ref{Th:GontcharoffMacintyre}]
Using \eqref{eq:maingrowthcondition} and Corollary \ref{Corollary:Polya}, we deduce from the assumptions of Theorem \ref{Th:GontcharoffMacintyre} 
that 
$$
f^{(mn+j)}(\sigma_j) =0
$$
for all sufficiently large $n$. 
It follows from Proposition \ref{Proposition:GontcharoffMacintyre} and the assumption $\tau(f)<\tau$ that $f(z)$ is a polynomial. 
\end{proof}

\begin{proof}[Proof of Proposition \ref{Prop:GontcharoffMacintyre}]
{\rm (a)} Assume $\Delta(\alpha)=0$: the $m\times m$ matrix 
$$
\Bigl(
\zeta^{k\ell}\rme^{\zeta^k\alpha\sigma_\ell} 
\Bigr)_{0\le k,\ell\le m-1}
$$
has rank $<m$. There exists $c_0,c_1,\dots,c_{m-1}$ in $\C$, not all zero, such that the function 
 $$
 f(z)=c_0\rme^{\alpha z}+c_1\rme^{\zeta\alpha z}+\cdots+c_{m-1}\rme^{\zeta^{m-1}\alpha z}
 $$
 satisfies 
 $$
f^{(j)}(\sigma_j) =0 \text{ for $j=0,1,\dots,m-1$}.
$$
Since $f^{(m)}(z)=\alpha^m f(z)$, one deduces
$$
f^{(mn+j)}(\sigma_j) =0 \text{ for $j=0,1,\dots,m-1$ and $n\ge 0$}.
$$ 
{\rm (b)}
 Assume $\Delta(1)\not=0$.
 For $j=0,1,\dots,m-1$, there exists a unique $m$--tuple of complex numbers $(c_{j0},c_{j1},\dots,c_{j,m-1})$ such that the function 
$$
\varphi_j(z)=\sum_{k=0}^{m-1} c_{jk} \rme^{\zeta^k z}
$$ 
satisfies
$$
\varphi_j^{(\ell)}(\sigma_\ell)= \delta_{j\ell}\quad \hbox{for}\quad 0\le \ell\le m-1.
$$
For $j=0,1,\dots,m-1$, the function $\varphi_j(z)$ has exponential type $1$ and is a solution of the differential equation $\varphi_j^{(m)}=\varphi_j$. 
Let $a_0,a_1,\dots,a_{m-1}$ in $\C$. Define 
$$
f(z)=a_0\varphi_0(z)+a_1\varphi_1(z)+\cdots+a_{m-1}\varphi_{m-1}(z).
$$
We have
 $$
f^{(mn+j)}(\sigma_j) = a_j \text{ for $j=0,1,\dots,m-1$ and $n\ge 0$}.
$$
Assume now $\tau>1$: according to Proposition \ref{Proposition:GontcharoffMacintyre}, for $a_0=a_1=\cdots=a_{m-1}=0$, the unique solution of exponential type $<\tau$ is $f=0$. The unicity follows. 
\end{proof}

\begin{proof}[Proof of Corollary \ref{Cor:GontcharoffMacintyre}]
In case $\sigma_0=1$, $\sigma_1=\sigma_2=\cdots=\sigma_{m-1}=0$, the determinant $\Delta(t)$ is 
$$
\det
\begin{pmatrix}
 \rme^{t} & 1 & 1& \cdots & 1 
\\
\rme^{ \zeta t} &\zeta &\zeta^2 & \cdots & \zeta^{m-1} 
\\
 \vdots&\vdots&\ddots&\vdots &\vdots 
 \\
\rme^{\zeta^{m-1} t} &
\zeta^{m-1} &\zeta^{2(m-1)} & \cdots & \zeta^{(m-1)^2} 
\end{pmatrix}.
$$ 
This determinant is invariant under the transformation $t\mapsto\zeta t$; hence $\Delta(t) $ is a nonzero constant times 
$$
\rme^t +\rme^{\zeta t}+\cdots+\rme^{\zeta^{m-1}t} =
m\sum_{n\ge 0} 
\frac{t^{nm}}{(nm)!} \cdotp
$$
 Now Corollary \ref{Cor:GontcharoffMacintyre} follows from Theorem \ref{Th:GontcharoffMacintyre} with $\tau=\tau_m/|s_1-s_0|$. 
 \end{proof}
 
As pointed out by
 \cite[p.~12]{MR0059366}, 
a special case of the results of 
 \cite{zbMATH02563461} 
is that an entire function of exponential type $<\tau_m$ satisfying 
$$
f^{(n)}(0) =0 \text{ for $n\equiv 0 \bmod m$ }
\text{ and }\;
f^{(n)}(1)=0 \text{ for $n\not \equiv 0 \bmod m$}
$$ 
is a polynomial. 
A.J.~Macintyre remarks that $\tau_m$ is approximately $m/e$ when $m$ is large; he suggests an analogy with Taylor's series which may be considered as the limiting case with $m=\infty$. 
For Corollary \ref{Cor:GontcharoffMacintyre}, when $m$ is large, the assumption \eqref{eq:maingrowthcondition} implies the assumption on $\tau(f)$. Hence Proposition \ref{Prop:Polya} can be viewed as the limiting case of Corollary \ref{Cor:GontcharoffMacintyre}.

 \subsection{Proof of Theorem \ref{Th:2PointsAlln}}\label{SS:Whittaker}
 
 The proof of Theorem \ref{Th:2PointsAlln} relies on the following result \cite[Corollary of Theorem 7 p.~468]{zbMATH02543645}: 
 
\begin{proposition}\label{Proposition:Whittaker}
If an entire function $f$ of exponential type $\tau(f) <1$ satisfies
$$
 f^{(n)}(0)f^{(n)}(1)=0
$$ 
for all sufficiently large $n$, then $f$ is a polynomial. 
\end{proposition}

As pointed out in a note added in proof of \cite[p.~469]{zbMATH02543645}, 
 \cite{zbMATH02563461} 
 proved this result earlier under the stronger assumption $\tau(f)<1/e$. 
 
\begin{proof}[Proof of Theorem \ref{Th:2PointsAlln}] 
Since $f$ satisfies \eqref{eq:maingrowthcondition}, the assumption of Corollary \ref{Corollary:Polya} is satisfied, hence
$|f^{(n)}(s_j)|<1$ for $n$ sufficiently large and $j=0,1$. Let $n$ be sufficiently large. 
One at least of the three numbers $ f^{(n)}(s_0)$ $ f^{(n)}(s_1)$, $ f^{(n)}(s_0) f^{(n)}(s_1)$ is an integer of absolute value less than $1$, hence it vanishes and therefore the product $ f^{(n)}(s_0) f^{(n)}(s_1)$ vanishes. We apply Proposition \ref{Proposition:Whittaker} to the function 
$$
\fhat(z)= f\bigl( s_0+z( s_1-s_0)\bigr), 
$$
the exponential type of which is $| s_1-s_0|\tau(f)<1$. 

This completes the proof of Theorem \ref{Th:2PointsAlln}.
\end{proof}
 

 \vskip 2truecm plus .5truecm minus .5truecm 
 
 \null
 \hfill
\vbox{\hbox{{\sc Michel WALDSCHMIDT,} Sorbonne Universit\'e, Faculté Sciences et Ingenierie} 
	\hbox{CNRS, Institut Mathématique de Jussieu Paris Rive Gauche IMJ-PRG, 75005 Paris, France }
	\hbox{E-mail: \url{michel.waldschmidt@imj-prg.fr}} 
	\hbox{Url: \url{http://www.imj-prg.fr/~michel.waldschmidt}}	
}	

\end{document}